\documentclass[a4paper,11pt,reqno]{amsart}

\usepackage[utf8]{inputenc}
\usepackage[T1]{fontenc}
\usepackage{amsfonts,amssymb}
\usepackage{comment}
\usepackage{amsmath}
\usepackage{graphicx}
\usepackage{array}
\usepackage{enumerate}
\usepackage{url}
\usepackage[all]{xy}
\usepackage{tikz}
\usetikzlibrary{arrows,shapes,positioning,calc,3d,decorations.pathreplacing,decorations.markings,cd,patterns}
\tikzstyle{dot}=[shape=circle,fill=black,inner sep=1.5pt]
\tikzstyle{root}=[ultra thick,->]
\tikzstyle{weight}=[orange,ultra thick,->]
\tikzstyle{weightdot}=[dot,orange]
\tikzstyle{wall}=[blue, thin, dashed]
\tikzstyle{auxwall}=[blue!20, thin, dashed]
\tikzstyle{->-}=[postaction={decorate,decoration={markings,mark=at position .5 with {\arrow{to}}}}]
\tikzstyle{-<-}=[postaction={decorate,decoration={markings,mark=at position .5 with {\arrowreversed{to};}}}]
\tikzstyle{--}=[double,double distance=1pt]
\usepackage{xargs}
\usepackage{stmaryrd}
\usepackage{hyperref}
\usepackage{multirow}
\usepackage{mathrsfs}

\usepackage{empheq}
\numberwithin{equation}{section}

\usepackage{amsthm}
\theoremstyle{plain}
\newtheorem{introtheorem}{Theorem}
\newtheorem{introcorollary}[introtheorem]{Corollary}
\newtheorem{introquestion}[introtheorem]{Question}
\newtheorem{theorem}{Theorem}[section]
\newtheorem{lemma}[theorem]{Lemma}
\newtheorem{corollary}[theorem]{Corollary}

\newtheorem{proposition}[theorem]{Proposition}
\newtheorem{observation}[theorem]{Observation}

\theoremstyle{definition}
\newtheorem{definition}[theorem]{Definition}

\newtheorem{example}[theorem]{Example}
\theoremstyle{remark}
\newtheorem*{remark}{Remark}

\newcommand{\Z}{\mathbb{Z}}
\newcommand{\N}{\mathbb{N}}

\newcommand{\calA}{\mathcal{A}}

\newcommand{\calU}{\mathcal{U}}

\newcommand{\calS}{\mathcal{S}}
\newcommand{\calE}{\mathcal{E}}

\newcommand{\VR}{\operatorname{VR}}
\newcommand{\Sd}{\operatorname{Sd}}

\newcommand{\defeq}{\mathrel{\mathop{:}}=}
\newcommand{\eqdef}{=\mathrel{\mathop{:}}}
\newcommand{\abs}[1]{\lvert #1 \rvert}

\newcommand{\colim}{\operatorname{colim}}

\newcommand{\Sym}{\operatorname{Sym}}

\newcommand{\id}{\operatorname{id}}
\newcommand{\calK}{\mathcal{K}}
\newcommand{\calCB}{\mathcal{CB}}
\renewcommand{\setminus}{\smallsetminus}

\newcommand{\scx}{C}
\newcommand{\sset}{E}
\renewcommand{\bar}{\overline}

\newcommand{\newcomment}[4]{%
\newcounter{#2counter}
\expandafter\newcommand\csname #1\endcsname[1]{%
\refstepcounter{#2counter}%
{\color{#4}(#3\arabic{#2counter})}\marginpar{\scriptsize\raggedright\textbf{\color{#4}(#2 \arabic{#2counter}):} ##1}%
}}
\reversemarginpar

\definecolor{darkgreen}{rgb}{0,0.6,0}
\newcomment{sw}{Stefan}{S}{darkgreen}
\newcomment{dc}{Dorian}{D}{blue}
\newcomment{jq}{Jose}{J}{orange} 

\begin{document}
\title[Pragmatic finiteness properties of locally compact groups]{Pragmatic finiteness properties of\\locally compact groups}
\date{\today}


\author[D.~Chanfi]{Dorian Chanfi}
\address{Mathematisches Institut, JLU Gießen, Arndtstr.\ 2, D-35392 Gießen, Germany}
\thanks{S.W.\ was supported through the DFG Heisenberg project WI 4079/6.}
\email{stefan.witzel@math.uni-giessen.de}


\author[S.~Witzel]{Stefan Witzel}
\address{Mathematisches Institut, JLU Gießen, Arndtstr.\ 2, D-35392 Gießen, Germany}
\thanks{D.C.\ was supported through the DFG project WI 4079/7 and a fellowship of the Humboldt Foundation.}
\email{dorian.chanfi@math.uni-giessen.de}

\begin{abstract}
  We compare finiteness properties of locally compact groups that generalize the properties of being compactly generated and of being compactly presented. Three such families of properties have been proposed: Abels--Tiemeyer's type $C_n$, coarse $(n-1)$-connectedness, and Castellano--Corob-Cook's type $F_n$. The first was defined for locally compact groups, the second can be defined for general topological groups, while the third was defined only for tdlc groups. We prove that all three families lead to the same notion for locally compact groups. This justifies working with these properties despite the fact that it is still unclear in which sense they describe finiteness properties of free classifying spaces for locally compact groups. Various parts of the arguments are well-known to various experts. By putting them together we hope to clarify the literature.
\end{abstract}

\maketitle

A discrete group $G$ is of type $F_n$ if there it admits a model for $EG$, i.e.\ a contractible CW-complex on which $G$ acts freely by permuting the cells, whose $n$-skeleton is cocompact. A group is of type $F_1$ or of type $F_2$ iff it is finitely generated respectively finitely presented. If $G$ is a locally compact group, it is not so clear how to define what a model for $EG$ is via a universal property that  leads to a notion of finiteness properties. In the absence of such a notion of finiteness properties of locally compact groups, various pragmatic definitions have been proposed that start with an (obviously reasonable) assumption (that compact groups should enjoy all finiteness properties, that a certain space should be a model for $EG$, or both) and base the definition on this.

These properties are: being of type $C_n$ \cite{AbelsTiemeyer97}, being coarsely $(n-1)$-connected (extrapolating \cite{Alonso94}), and, for totally disconnected groups, being of type $F_n$ in the sense of \cite{CastellanoCorobCook20}. Our main result is:

\begin{introtheorem}\label{thm:main}
  Let $G$ be a locally compact group and let $G^0$ be the connected component of the identity. The following conditions are equivalent:
  \begin{enumerate}
    \item $G$ is of type $C_n$ in the sense of Abels--Tiemeyer.\label{item:cn}
    \item $G$ is coarsely $(n-1)$-connected (for the coarse structure $\calK$).\label{item:cc}
    \item $G/G^0$ is of type $F_n$ in the sense of Castellano--Corob-Cook.\label{item:fn}
  \end{enumerate}
\end{introtheorem}

We take this theorem as a justification to call a locally compact group $G$ \emph{of type $F_n$} if it satisfies these equivalent properties. As usual we say that $G$ is \emph{of type $F_\infty$} if it is of type $F_n$ for all $n \in \N$. An immediate consequence of the theorem, and really the main part of the proof, is:

\begin{introcorollary}\label{cor:main}
  A locally compact group $G$ is of type $F_n$ if and only if $G/G^0$ is of type $F_n$. In particular, connected groups are of type $F_\infty$.
\end{introcorollary}

The proof combines methods and results from \cite{AbelsTiemeyer97, CastellanoCorobCook20, BuxHartmannQuintanilha, Tiemeyer94}.
Theorem~\ref{thm:main} for tdlc groups is known to experts. Ingredients to Corollary~\ref{cor:main} are less well-known. We hope that assembling the arguments will help to clarify the situation and push research on finiteness properties of locally compact group by allowing researchers to work with their preferred definition.

The theorem suggests that any reasonable notion of finiteness properties for locally compact groups leads to the same notion. For locally compact $\sigma$-compact groups we think this is true, but for groups that are not $\sigma$-compact it is less clear. The reason is that coarse $(n-1)$-connectedness (explicitly) and type $C_n$ (implicitly) rely on a choice of coarse structure. The structure $\calK$ is the one induced by relatively compact sets, while Rosendal defines a coarse structure $\calCB$ that seems more appropriate for non-locally compact groups (such as Fréchet spaces). So for general topological groups one may want to study the properties of being coarsely $(n-1)$-connected for $\calCB$. One should note, however, that this property does not coincide with usual finiteness properties for uncountable discrete groups: $\Sym(\N)$ as a discrete group is not finitely generated but with respect to $\calCB$ it is bounded and thus $(n-1)$-connected for all $n$. For locally compact $\sigma$-compact groups the coarse structures $\calK$ and $\calCB$ coincide so there is no doubt about which structure to consider. This will be discussed in some detail in Sections~\ref{sec:coarse_connectivity} and~\ref{sec:filtration} below.

While it appears that the equivalent notions in Theorem~\ref{thm:main} of being of type $F_n$ identify the right class of locally compact groups it remains elusive what these conditions actually mean: the actual definition should concern the existence of a universal free $G$-space that is cocompact up to a kind of degree $n$, or possibly the existence of free, cocompact $G$-space that is $n$-universal (in a sense generalizing \cite[§19]{Steenrod99}). In particular, it should entail finiteness of continuous cohomology. While we have examples of what ought to be a universal $G$-space, such as \cite{Milnor56} and the one used in defining $C_n$, we seem to be lacking a categorical characterization of a universal free $G$-space that consolidates the topology of $G$ with a notion of $n$-skeleton. (For non-free actions of topological groups this is often less intricate, see for example \cite{Lueck05}). For this reason we formulate:

\begin{introquestion}
  What is the right notion of a universal $G$-space and what kind of finiteness does type $F_n$ characterize?
\end{introquestion}

\subsection*{Acknowledgments} The authors thank Kai-Uwe Bux, Ilaria Castellano, Elisa Hartmann, Tobias Hartnick, José Quintanilha, and Roman Sauer for helpful discussions. The motivation to assemble the arguments here comes from an Oberwolfach Mini Workshop \cite{WitzelOWR} and the second author would like to thank the participants of this workshop as well as the staff of the MFO.

\section{Compactness properties following Abels--Tiemeyer}

In this section we recall Abels and Tiemeyer's properties $C_n$, see \cite{AbelsTiemeyer97}. We first recall some general notions. Let $(H_i)_{i \in I}$ be a directed system of groups. We say that the system is \emph{essentially trivial} if for every $i \in I$ there is a $j \in I$, $j \ge i$ such that $H_i \to H_j$ is trivial. An essentially trivial system satisfies $\colim H_i = \{1\}$ but the converse is not true: in general the morphism $H \to H'$ under which an element $h \in H$ is mapped trivially may depend on $h$. Let $(X_i)_{i \in I}$ be a directed system of topological spaces. We say that the system is \emph{essentially $n$-connected} if for every $i \in I$ there is a $j \in I$, $j \ge i$ such that for all $k \le n$ the map $\pi_k(X_i \to X_j)$ is trivial (essential $0$-connectedness guarantees that the basepoint does not matter).

Let $G$ be a locally compact group. For a set $T$ let $ET$ denote the free simplicial set on $T$: its set of $k$-simplices $T[k] = T^{k+1}$ consists of tuples and there are the obvious face and degeneracy maps. If $K \subseteq G$ is a subset then $EK$ is a sub-simplicial set of $EG$, as are the translates $g \cdot EK$ for $g \in G$, and their union $G\cdot EK = \bigcup_{g \in G} g \cdot EK$. Abels and Tiemeyer think of $EG$ as a free classifying space and of $G \cdot EK$ with $K$ compact as a cocompact subset. With this in mind they define $G$ to be \emph{of type $C_n$} if the directed system $(\abs{G \cdot EK})_{K}$ is essentially $(n-1)$-connected, where $K$ ranges over (relatively) compact subsets.

In general it is not clear how to replace $\abs{EG}$ by a different space (such as a complex on which $G$ acts), but Abels and Tiemeyer do prove that it may be replaced by $\abs{ET}$ for a proper $G$-space $T$ \cite[Theorem~3.2.2]{AbelsTiemeyer97}:

\begin{proposition}\label{prop:at_space}
  Let $G$ act properly on a locally compact space $T$. Let $ET$ be the free simplicial set on $T$ and consider the filtration $(|G \cdot EK|)_{K \subseteq T \text{ compact}}$ of its geometric realization $|ET|$. Then $G$ is of type $C_n$ if and only if $(\pi_i( (|G \cdot EK|))_K$ is essentially trivial for $i < n$.
\end{proposition}

\section{Finiteness properties for tdlc groups following Castellano--Corob-Cook}

In their article \cite{CastellanoCorobCook20} Castellano and Corob Cook introduce finiteness properties for totally disconnected locally compact groups making use of the fact that these can act properly (though not freely) on CW complexes. Let $G$ be a totally disconnected locally compact group. A topological space $X$ is called a \emph{discrete $G$-CW complex} if there exists a filtration $X = \bigcup_{n \ge 0} X^{(n)}$ of $X$ by $G$-invariant closed subspaces such that
\begin{enumerate}
    \item $G$ acts on $X^{(0)}$ with open stabilizers,
    \item for each $n \ge 1$, there exists a $G$-space $\Lambda_n$ on which $G$ acts with open stabilizers and $G$-equivariant maps $f_n: S^{n-1} \times \Lambda_n \to X^{(n-1)}$ and $\hat f_n: B^n \times \Lambda_n \to X^{(n)}$ such that the diagram
    \begin{center}
        \begin{tikzcd}
            S^{n-1} \times \Lambda_n \arrow[d] \arrow[r, "f_n"] & X^{(n-1)} \arrow[d] \\
            B^n \times \Lambda_n \arrow[r, "\hat f_n"] & X^{(n)} 
        \end{tikzcd},
    \end{center}
    where $S^{n-1}$ (resp. $B^n$) is the Euclidean unit $(n-1)$-sphere (resp. $n$-ball) endowed with the trivial $G$-action, is cocartesian,
    \item the topology on $X$ is the direct limit of the topologies on the $X^{(n)}$.
\end{enumerate}
The complex $X$ is said to be \emph{proper discrete} if the $X^{(n)}$ and $\Lambda_n$ may be so chosen that the aforementioned actions are with compact open stabilizers.
A totally disconnected locally compact group $G$ is then said to be \emph{of type $F_n$} if there exists a contractible proper discrete $G$-CW complex $X$ such that $G$ acts cocompactly on $X^{(n)}$.

The authors prove Brown's criterion for their properties, the following is a combination of Proposition~3.13, Theorem~4.7, and Theorem~4.10 of \cite{CastellanoCorobCook20}:

\begin{theorem}\label{thm:cc_brown}
  Let $G$ be totally disconnected locally compact group and let $X$ be an $(n-1)$-connected discrete $G$-CW complex. Assume that the stabilizer of a $k$-cell is of type $F_{n-k}$. Let $(X_k)_k$ be a filtration by $G$-invariant subcomplexes such that $X_k^{(n)}$ is $G$-cocompact for each $k$. Then $G$ is of type $F_n$ if and only if $(\pi_i(X_k))_k$ is essentially trivial for $i < n$.
\end{theorem}

\section{Coarse connectivity}\label{sec:coarse_connectivity}
Alonso \cite{Alonso94} proved that finiteness properties of discrete groups are coarse connectivity properties and thus coarse invariants (he only stated it for finitely generated groups). Using work of Roe \cite{Roe03} and Rosendal \cite{Rosendal22} one can define coarse connectivity properties for arbitrary topological groups which, by analogy, are candidate definitions for finiteness properties of groups.

If $T$ is a set we let $\scx T$ denote the abstract simplicial complex that models the simplex on $T$, i.e.\ its vertex set is $T$ and the abstract simplices are finite subsets of $T$. If $T$ is a metric space and $r > 0$, the \emph{Vietoris--Rips complex} of $T$ with radius $r$ is the subcomplex 
\[
  \VR_r(T) \defeq \{s \in \scx T \mid \forall v,w \in s, \ d(v,w) \le r\}
\]
of $\scx T$ (it is more usual to require the inequality to be strict, but the difference will be immaterial in what follows).

There are two generalizations of metric spaces, formalizing behaviour close to zero and close to infinity, respectively. To introduce them, for subsets $U,V \subseteq T \times T$ we introduce the notations
\begin{align*}
  U^{-1} &= \{ (v,w) \in T \times T \mid (w,v) \in U\} &\text{and}\\
  U \circ V &= \{ (u,w) \in T \times T \mid \exists v \in T\ (u,v) \in U, (v,w) \in V\}\text{.}
\end{align*}
The structure of a \emph{uniform space} \cite[Définition~II.1.1]{BourbakiTG15} on $T$ consists of subsets $\calU \subseteq T \times T$ satisfying the following conditions:
\begin{enumerate}
  \item Supersets and finite intersections of elements in $\calU$ are in $\calU$.\label{item:filter}
  \item Every $U \in \calU$ contains the diagonal $\{(v,v) \mid v \in T\}$.
  \item If $U \in \calU$ then $U^{-1} \in \calU$.
  \item For $V \in \calU$ there exists a $U \in \calU$ with $U \circ U \subseteq V$.
\end{enumerate}
The structure of a \emph{coarse space} \cite[Definition~2.3]{Roe03} on $T$ consists of subsets $\calE \subseteq T \times T$ satisfying the following conditions:
\begin{enumerate}
  \item Subsets and finite unions of elements in $\calE$ are in $\calE$.\label{item:ideal}
  \item The diagonal $\{(v,v) \mid v\in T\}$ is in $\calE$.
  \item If $U \in \calU$ then $U^{-1} \in \calU$.
  \item For $U,V \in \calE$ also $U \circ V \in \calE$.
\end{enumerate}
A map $f \colon X \to Y$ of coarse spaces is \emph{bornologous} if for $U \in \calE_X$ the set $(f \times f)(U) \in \calE_Y$.

If $\calS$ is either a uniform or a coarse structure on $T$ and $U \in \calS$ we define the Vietoris--Rips complex
\[
  \VR_U(T) = \{s \in \scx T \mid\forall v,w \in s\ (v,w) \in U\}
\]
(which only depends on $U \cap U^{-1}$).
Condition~\eqref{item:filter} asserts that $(\VR_U(T))_{U \in \calU}$ is directed upwards and downwards:
\begin{center}
\begin{tikzcd}
                          & \VR_{U \cup V}(T)                                   &                           \\
\VR_U(T) \arrow[ru, hook] &                                                     & \VR_V(T) \arrow[lu, hook] \\
                          & \VR_{U \cap V}(T) \arrow[ru, hook] \arrow[lu, hook] &                          
\end{tikzcd}
\end{center}

We say that a coarse space $(T,\calU)$ is \emph{coarsely $n$-connected} if for every $U \in \calU$ there exists a $V \in \calU$ containing $U$ such that the maps $\pi_i(\VR_U(T) \hookrightarrow \VR_V(T))$ induced between homotopy groups are trivial for $i \le n$.

Let $(X,\calE_X)$ and $(Y,\calE_Y)$ be coarse spaces. We say that maps $f,g \colon X \to Y$ are \emph{$U$-close}, where $U \in \calE_Y$ if $\{(f(x),g(x)) \mid x \in X\} \in U \cap U^{-1}$. We say that $Y$ is a \emph{coarse retract} of $X$ if there are bornologous maps $i \colon Y \to X$ and $r \colon X \to Y$ such that $\id \colon Y \to Y$ and $r \circ i \colon Y \to Y$ are $U$-close for some $U \in \calE_Y$. We say that $X$ and $Y$ are \emph{coarsely equivalent} if there are maps $i$ and $r$ as before satisfying in addition that $\id \colon X \to X$ and $i \circ r \colon X \to X$ are $U'$-close for some $U' \in \calE_X$. The following is a \cite[Theorem~8]{Alonso94} in this context:

\begin{proposition}
  Let $Y$ be a coarse retract of $X$. If $X$ is coarsely $n$-connected then so is $Y$. In particular, if $X$ and $Y$ are coarsely equivalent then $X$ is coarsely $n$-connected if and only if $Y$ is.
\end{proposition}

\begin{proof}
  Let $\iota \colon Y \to X$ and $r \colon X \to Y$ be such that $\id_Y$ and $r \circ \iota$ are $W'$-close for some $W' \in \calE_Y$.
  Let $U \in \calE_Y$ be arbitrary and let $V = (\iota \times \iota)(U)$. Since $X$ is coarsely $n$-connected, there is a $V' \in \calE_X$ containing $V$ such that, for each $i \le n$, the map $\pi_i(\VR_{V}(X)) \to  \pi_i(\VR_{V'}(X))$ induced by inclusion is trivial. Let $U' = (r \times r)(V')$. Then, for each $i \le n$, the map $(r \circ \iota)_* \colon \pi_i(\VR_U(Y)) \to \pi_i(\VR_{U'}(Y))$ factors through the aforementioned map and is thus trivial. The following lemma shows that, setting $W = U' \circ W'$, the maps $(r \circ \iota)_*,\id_* \colon \VR_U(X) \to \VR_W(X)$ are homotopic, showing that $(\VR_U(Y))_U$ is essentially $n$-connected.
\end{proof}

\begin{lemma}
  Let bornologous maps $f,g \colon X \to Y$ be $W$-close for $W \in \calE_Y$. Let $U \in \calE_X$ be arbitrary and put $V = (f,f)(U) \cup (g,g)(U)$. Then $f_*,g_* \colon \VR_U(X) \to \VR_{V \circ W}(Y)$ are homotopic.
\end{lemma}

\begin{proof}
  We want to show that if $s \in \VR_U(X)$ is a simplex, then the simplex $f(s) \cup g(s) \in \scx Y$ lies in $\VR_{V \circ W}(Y)$. If $s \in \VR_U(X)$ is a simplex then $(v,w) \in U$ for all vertices $v$ and $w$ in $s$. It follows that \[(f(v),f(w)),(g(v),g(w)) \in V\] showing that $f(s),g(s) \in \VR_V(X)$. Finally $(f(v),g(v)),(f(w),g(w)) \in W$ by assumption. It follows that $(f(v),g(w)),(g(v),f(w)) \in V \circ W$, showing that $f(s) \cup g(s) \in \VR_{V \circ W}(Y)$.
\end{proof}

Let now $G$ be a topological group. It carries a canonical uniform structure generated by the sets $U_V = \{(x,y) \in G \times G \mid x^{-1}v \in V\}$ where $V$ ranges over identity neighborhoods in $G$. A coarse structure $\calK$ is generated by the sets $U_C = \{(x,y) \in G \mid x^{-1}y \in C\}$ where $C$ ranges over relatively compact sets. Rosendal \cite[Section~2]{Rosendal22} defines a canonical coarse structure $\calCB$ on $G$ by declaring that $U \in \calE$ if and only if $\{ x^{-1}y \mid (x,y) \in U\}$ is bounded for every continuous left-invariant pseudo-metric on $G$.

We say that a topological group is \emph{coarsely $n$-connected} if $(X,\calCB)$ is coarsely $n$-connected. If $G$ is locally compact $\sigma$-compact then $\calK = \calCB$ \cite[Section~2.4]{Rosendal22}.

Note that the choice over $\calCB$ rather than $\calK$ has surprising effects even for uncountable discrete groups:

\begin{example}\label{ex:symmetric_group}
  Let $\Sym(\N)$ be equipped with the discrete topology. According to \cite[Example~2.26]{Rosendal22} $(\Sym(\N),\calCB)$ is bounded, hence coarsely $n$-connected for every $n$. Of course $(\Sym(\N),\calK)$ is not even coarsely $0$-connected.
\end{example}

\section{Filtrations}\label{sec:filtration}

As we have seen in the previous sections, there are various filtrations by cocompact spaces associated to a locally compact group $G$. In defining these filtrations there are (at least) three different aspects in each of which one can choose among two flavors, so there is a total of eight variants. The first aspect concerns the coarse structure, the obvious choices being $\calK$ and $\calCB$. The second aspect concerns the kind of complex, namely Čech or Vietoris--Rips. The third aspect is whether to use simplicial complexes or simplicial sets. The purpose of this section is to see that among these choices the choice of the coarse structure is the only one that affects $n$-connectedness.

Once the coarse structure is fixed, neither the group structure nor the topology play a role any more, so we consider a coarse space $(X,\calE)$ from now on.

\begin{definition}
  For $U \in \calE$ the \emph{Vietoris--Rips subsets} of diameter $U$ is
  \[
    \VR_U = \{A \subseteq X \mid \forall a,b \in A\ (a,b) \in U\}.
  \]
  The set of \emph{Čech subsets} of radius $U$ is
  \[
   \check{C}_U = \{A \subseteq X \mid \exists a \in X\ \forall b \in A\ (a,b) \in U\}.
  \]
\end{definition}

We make a few basic observations:

\begin{observation}\label{obs:vr_cech}
  \begin{enumerate}
    \item $\VR_U$ and $\check{C}_U$ are closed under passage to subsets.
    \item $\VR_U \subseteq \VR_V$ and $\check{C}_U \subseteq \check{C}_V$ for $U,V \in \calE$ with $U \subseteq V$.
    \item $\VR_U \subseteq \check{C}_{U}$ for $U \in \calE$.
    \item $\check{C}_U \subseteq \VR_{U \circ U}$ for $U \in \calE$.
  \end{enumerate}
\end{observation}

We write $\scx X$ for the abstract simplicial complex that models the simplex on $X$, that is, $\scx X$ is the set of finite subsets of $X$. We write $\sset X$ for the simplicial set that models the simplex on $X$, that is, the set $\sset X[n]$ of $n$-simplices consists of $(n+1)$-tuples of elements of $X$ with the obvious face and degeneracy maps.

\begin{definition}
  Let $\calA$ be a set of subsets of $X$ that is closed under taking subsets (such as $\VR_U$ or $\check{C}_U$). The associated simplicial subcomplex of $\scx X$ is
  \[
    \scx(\calA) = \{ A \in \calA \mid A \text{ finite}\}.
  \]
  The associated sub-simplicial set $\sset(\calA)$ of $\sset X$ is given by
  \[
    \sset(\calA)[n] = \{ (a_0,\ldots,a_n) \mid \exists A \in \calA\ a_0,\ldots,a_n \in A\} = \bigcup \{A^{n+1} \mid A \in \calA\}.
  \]
\end{definition}

In summary, for each coarse structure $\calE$ we obtain four filtrations:
\begin{enumerate}
  \item $(\abs{\scx(\VR_U)})_{U \in \calE}$ Vietoris--Rips simplicial complexes,
  \item $(\abs{\scx(\check{C}_U)})_{U \in \calE}$ Čech simplicial complexes,
  \item $(\abs{\sset(\VR_U)})_{U \in \calE}$ Vietoris--Rips simplicial sets,
  \item $(\abs{\sset(\check{C}_U)})_{U \in \calE}$ Čech simplicial sets.
\end{enumerate}

 Two of these we have already encountered:

\begin{lemma}\label{lem:comparison}
  \begin{enumerate}
    \item The filtration $(\abs{\scx(\VR_U)})_{U \in \calE}$ is the filtration used to define coarse connectivity properties for $\calE$.
    \item The filtration $(\abs{\sset(\check{C}_U)})_{U \in \calK}$ is the filtration used to define compactness properties $C_n$.
  \end{enumerate}
\end{lemma}

\begin{proof}
  For $(\scx(\VR_U))_{U \in \calE}$ the claim is clear. For $\sset(\check{C}_U)$ we unravel the definitions: a set $U \in \calK$ is of the form $\{(g,g\cdot h) \mid g \in G, h \in C\}$ for some relatively compact $C \subseteq G$. A tuple $(v_0,\ldots,v_k)$ is a simplex in $\sset(\check{C}_U)$ if and only if there is an $a \eqdef g \in G$ such that $(a,v_i) \in U$, that is $a^{-1}v_i \eqdef h_i \in C$. The simplices of this form are precisely those contained in $G \cdot EC$.
\end{proof}

\begin{proposition}\label{prop:equivalent_filtrations}
  Let $(X,\calE)$ be a coarse space and let $n \in \N$. If one of the filtrations $(\abs{\scx(\VR_U)})_{U \in \calE}$, $(\abs{\scx(\check{C}_U)})_{U \in \calE}$, $(\abs{\sset(\VR_U)})_{U \in \calE}$, $(\abs{\sset(\check{C}_U)})_{U \in \calE}$ is essentially $n$-connected then they all are.
\end{proposition}

\begin{proof}
  It follows from Observation~\ref{obs:vr_cech} that $\scx(\VR_U) \subseteq \scx(\check{C}_U) \subseteq \scx(\VR_{U \circ U})$ from which it follows that one of the filtrations $(\abs{\scx(\VR_U)})_{U \in \calE}$ is essentially $n$-connected if and only if $(\abs{\scx(\check{C}_U)})_{U \in \calE}$ is. Similarly $\sset(\VR_U) \subseteq \sset(\check{C}_U) \subseteq \sset(\VR_{U \circ U})$ implies that $({\sset(\VR_U)})_{U \in \calE}$ is essentially $n$-connected if and only if $({\sset(\check{C}_U)})_{U \in \calE}$.

  Finally, for $A \subseteq X$  there is an obvious map $\abs{\sset A} \to \abs{\scx A}$ that is a homotopy equivalence by \cite[Proposition 15]{AntolinVillareal}. From this it follows that $(\abs{\scx(\VR_U)})_{U \in \calE}$ is essentially $n$-connected if and only if $(\abs{\sset(\VR_U)})_{U \in \calE}$ is; similarly for $(\abs{\scx(\VR_U)})_{U \in \calE}$ and $(\abs{\sset(\VR_U)})_{U \in \calE}$.
\end{proof}

\begin{corollary}\label{cor:cn_coarse_connected}
  A locally compact group $G$ is of type $C_n$ if and only if it is coarsely $(n-1)$-connected for the coarse structure $\calK$.
\end{corollary}

\begin{proof}
  By Lemma~\ref{lem:comparison} $G$ is of type $C_n$ if and only if $(\abs{\sset(\check{C}_U)})_{U \in \calK}$ is essentially $(n-1)$-connected. By Proposition~\ref{prop:equivalent_filtrations} this is the case if and only if $(\abs{\scx(\VR_U)})_{U \in \calK}$ is essentially $(n-1)$-connected, which by definition is equivalent to $G$ being coarsely $(n-1)$-connected for $\calK$.
\end{proof}


\section{Connected groups}

In this section we prove Corollary~\ref{cor:main} in the following formulation:

\begin{theorem}\label{thm:connected}
  A connected, locally compact group $G$ is of type $C_n$ for all $n$.
\end{theorem}

It reduces to the case of connected Lie groups, which is due to Tiemeyer \cite[Satz~5.0.1]{Tiemeyer94}:

\begin{proposition}\label{prop:lie_grps}
    Let $G$ be a connected Lie group. Then $G$ is of type $C_n$ for all $n \in \N$.
\end{proposition}
    
\begin{proof}[Proof of Theorem~\ref{thm:connected} using Proposition~\ref{prop:lie_grps}]
  The solution of Hilbert's fifth problem by Gleason, Montgomery, Zippin, Yamabe \cite[Theorem, p.~175]{MontgomeryZippin55} implies that there is an extension
  \[
    1 \to K \to G \to Q \to 1
  \]
  where $K$ is compact and $Q$ is connected Lie. Applying Proposition~\ref{prop:at_space} to the proper action of $G$ on $Q$ we see that $G$ is $C_n$ if and only if $Q$ is $C_n$. Since $Q$ is $C_\infty$ by Proposition~\ref{prop:lie_grps}, so is $G$.
\end{proof}

The rest of the section is devoted to the proof of Proposition~\ref{prop:lie_grps}. It follows \cite{Tiemeyer94} but is carried out in terms of coarse connectedness, i.e.\ we show that $G$ is coarsely $n$-connected for all $n$ for the coarse structure $\calK$.

We use the following notation and results for simplicial complexes, see for example \cite{Spanier66} and \cite[§3]{Jardine04}. 
For $n \in \N$, set $\Delta^n = C(\{0, \dots, n\})$ and $\partial \Delta^n = \Delta^n \setminus \{\{0, \dots, n\}\}$.
If $X$ is an (abstract) simplicial complex we write $\abs{X}$ for its realization. Similarly, if $f \colon X \to Y$ is a simplicial map of simplicial complexes, we write $\abs{f} \colon \abs{X} \to \abs{Y}$ for its realization. We denote the barycentric subdivision of $X$ by $\Sd^n X$; its vertex set consists of barycenters $b_\sigma$ of simplices $\sigma$ of $X$. Iterating, we obtain the $n$-fold barycentric subdivision $\Sd^n X$. We canonically identify $\abs{\Sd^n X}$ and $\abs{X}$. In order to obtain a canonical simplicial map $\gamma_X \colon \Sd X \to X$ we pick a total order on the vertices of $X$ and put $\gamma_X(b_\sigma) = \min \sigma$. Note that for any simplex $\tau = \{b_{\sigma_1}, \dots, b_{\sigma_r}\}$ of $\Sd X$ corresponding to a flag $\sigma_1 < \dots < \sigma_r$ of simplices of $X$, we have $\gamma_X(\tau) \subseteq \sigma_r$.
Also note that, if $f\colon X \to Y$ is a simplicial map that is injective and order-preserving on vertices, then we have $f \circ \gamma_X = \gamma_Y \circ \Sd(f)$.
Once a total order on the vertices of $X$ is fixed, we order the vertices of $\Sd X$ by $b_\sigma < b_\tau$ iff $\dim \sigma < \dim \tau$ or [$\dim \sigma = \dim \tau$ and $\min (\sigma \setminus \tau) < \min (\tau \setminus \sigma)$]. This way we get a map $\gamma^n_X \colon \Sd^n X \to X$. We will usually write $\gamma$ (resp. $\gamma^n$) instead of $\gamma_X$ (resp. $\gamma_X^n$) when there is no risk of confusion. Note that $\abs{\gamma^n} \simeq \id \colon \abs{X} \to \abs{X}$. We say that two maps $f,g \colon X \to Y$ are contiguous if for every simplex $\sigma$ of $X$ the set $f(\sigma) \cup g(\sigma)$ is a simplex of $Y$. If $\sigma$ is a simplex of $X$, we denote by $\langle \sigma \rangle$ the subcomplex consisting of $\sigma$ and its faces.

The group $G$ admits a maximal compact subgroup $C$ and the quotient $X = G/C$ is a contractible smooth manifold \cite[Théorème 1]{Borel}, \cite[Theorems~14.1.3,~14.3.11]{HilgertNeeb}. Moreover $X$ carries a $G$-invariant Riemannian metric \cite[Example IV.1.3]{KobayashiNomizu63}. Note that $X$ is then geodesically complete \cite[Theorem IV.4.5]{KobayashiNomizu63}, hence proper as a metric space by the Hopf-Rinow theorem. 
It follows from the Švarc-Milnor lemma \cite[Theorem 4.C.5, Proposition 4.C.8 (2')]{CornulierDeLaHarpe16} that the orbit map $\pi: G \to X$ defined by $\pi(g) = gC$ for all $g \in G$ is a quasi-isometry, and in particular a coarse equivalence.
By \cite[Theorem~3.5]{Hausmann95}, there exists $\varepsilon > 0$ such that $\VR_r(X)$ is contractible for all $0 < r \le \varepsilon$. We fix such an $\varepsilon$ for the remainder of the proof and we put $\varepsilon_i = (1-2^{-i-1})\varepsilon$ for each $i\in \N$.

If $Y$ is a finite simplicial complex and $f,g \colon Y \to \VR_r(X)$ are simplicial maps, we put $d_\infty(f,g) = \sup \{d_X(f(v),g(v)) \mid v \in Y^{(0)}\}$ where the vertices $f(v),g(v)$ are regarded as points in $X$. With this setup we can formulate the key lemma: every sphere in $\VR_{\varepsilon_{n-1}}(X)$ can be filled by a subdivided disk and the degree of subdivision is uniform in a small neighborhood when the filling is taken in $\VR_{\varepsilon_{n}}(X)$.

\begin{lemma}\label{lem:filling_in_nbd}
  Let $m,n \in \N$. For every simplicial map $f \colon \Sd^k \partial \Delta^m \to \VR_{\varepsilon_{n-1}}(X)$ there exists an $\ell \ge k$ such that for all $g \colon  \Sd^k \partial \Delta^m \to \VR_{\varepsilon_{n-1}}(X)$ with $d_\infty(f,g) < 1/2(\varepsilon_n - \varepsilon_{n-1})$ there is a simplicial map $\hat{g} \colon \Sd^\ell \Delta^m \to \VR_{\varepsilon_n}(X)$ with $\hat{g}|_{\Sd^\ell \partial \Delta^m} = g \circ  \gamma^{\ell-k}$.
\end{lemma}

\begin{proof}
  Since $\abs{\VR_{\varepsilon_{n-1}}(X)}$ is contractible, there exists a map $\bar{f} \colon \abs{\Delta^m} \to \VR_{\varepsilon_{n-1}}(X)$ extending $\abs{f}$. By \cite[Theorem 3.1]{Jardine04} there exists an $\ell \ge k$ and a simplicial map $\hat{f} \colon \Sd^\ell \Delta^m \to \VR_{\varepsilon_{n-1}}(X)$ such that $\hat{f}|_{\Sd^\ell \partial \Delta^m} = f \circ \gamma^{\ell - k}$.

  Now if any two maps $f,g \colon Y \to \VR_{\varepsilon_{n-1}}(X)$ satisfy $d_\infty(f,g) < \frac{1}{2}(\varepsilon_n - \varepsilon_{n-1})$ then they are contiguous as maps $Y \to \VR_{\varepsilon_n}(X)$. Thus a map $h \colon Y^{(0)} \to \VR_{\varepsilon_{n-1}}$ satisfying $h(v) \in \{f(v),g(v)\}$ is simplicial.

  So if $g$ is as in the lemma, defining $\hat{g}$ by
  \[
    \hat{g}(v) = \begin{cases}
      g \circ \gamma^{\ell - k}(v) & \text{if }v \in \Sd^\ell \partial \Delta^m\\
      \hat{f}(v) & \text{else},
    \end{cases}
  \]
  gives a simplicial map as required.
\end{proof}

The lemma will be combined with the following compactness statement. 

\begin{lemma}\label{lem:compactness}
  Let $(Y,y)$ be a pointed, finite connected simplicial complex, and let $x \in X$. Then the set $M$ of pointed simplicial maps $f \colon (Y,y) \to (\VR_{\varepsilon_n},x)$ is compact, regarded as a subset of the space of maps $Y^{(0)} \to X$ (a finite power of $X$).
\end{lemma}

\begin{proof}
  Since $Y$ is finite, it has finite diameter, say $d$. Thus, because $f$ is simplicial and $Y$ is connected, for every vertex $y' \in Y$, we have $d(f(y),f(y')) \le d \varepsilon_n$, showing that $M$ is bounded. It is closed because the conditions to be simplicial ($d_X(f(y_1),f(y_2)) \le \varepsilon_n$ for every edge $\{y_1,y_2\}$ of $Y$) and pointed ($f(y) = x$) are closed.
\end{proof}

Combining the two lemmas and using the homogeneity of $X$ we get that a $(m-1)$-sphere with $m \ge 2$ in $\VR_{\varepsilon_{n-1}}(X)$ can be filled by a subdivided disk in $\VR_{\varepsilon_n}(X)$ whose degree of subdivision is uniformly bounded.

\begin{lemma}\label{lem:filling}
    Let $m \in \N_{\ge 2}$, $n \in \N$. For every $k\in \N$, there exists $\ell \ge k$ such that, for every simplicial map $f \colon \Sd^k \partial \Delta^m \to \VR_{\varepsilon_{n-1}}(X)$, there is a map $\hat{f} \colon \Sd^\ell \Delta^m \to \VR_{\varepsilon_n}(X)$ with $\hat{f}|_{\Sd^\ell \partial \Delta^m} = f \circ  \gamma^{\ell-k}$. 
\end{lemma}

\begin{proof}
    Fix a base vertex $v \in \partial \Delta^m$ and a point $x \in X$. Then, because $\partial \Delta^m$ is connected, it follows from Lemma \ref{lem:compactness} that the set $M$ of pointed simplicial maps $f: (\Sd^k(\partial \Delta^m), v) \to (\VR_{\varepsilon_{n-1}}(X), x)$ is compact. Consequently, there exist maps $f_1, \dots, f_r \in M$ such that $$M \subset \bigcup_{i=1}^r \underbrace{M \cap B_{d_\infty}\left(f_i, \frac{1}{2}(\varepsilon_n - \varepsilon_{n-1})\right)}_{B_i}.$$
    By Lemma \ref{lem:filling_in_nbd}, for each $i \in \{1,\ldots,r\}$, there exists $\ell_i \ge k$ such that, for each $g \in B_i$, there exists $\hat g: \Sd^{\ell_i} \Delta^m \to \VR_{\varepsilon_n}(X)$ such that $\hat g|_{\Sd^{\ell_i}(\partial \Delta^m)} = g \circ \gamma^{\ell_i - k}$.
     It follows that, for each  $g \in M$, there exists $\hat g: \Sd^{\ell} \Delta^m \to \VR_{\varepsilon_n}(X)$ such that $\hat g|_{\Sd^{\ell}(\partial \Delta^m)} = g \circ \gamma^{\ell- k}$, where $\ell = \max(\ell_1, \dots, \ell_r)$. Finally, using the fact that any simplicial map $f \colon \Sd^k \partial \Delta^m \to \VR_{\varepsilon_{n-1}}(X)$ is a $G$-translate of an element of $M$, we get that the same conclusion holds for all simplicial maps from $\Sd^k \partial \Delta^m$ to $\VR_{\varepsilon_{n-1}}(X)$.
\end{proof}

Having made these preliminary observations, we turn to our main problem, namely that of proving that $G$ is of type $C_n$ for all $n \in\N$. By Corollary \ref{cor:cn_coarse_connected}, this amounts to proving that $G$ is coarsely $n$-connected for all $n \in \N$ or, equivalently by a previous observation, that $X$ is coarsely $n$-connected for all $n \in \N$, which is what we establish below. 
\begin{proof}[Proof of Proposition~\ref{prop:lie_grps}]
Let $n \in \N$ and $R > 0$. Our goal is to prove that there exists $S \ge R$ such that the inclusion $\VR_R(X) \to \VR_S(X)$ induces a trivial map between homotopy groups of degree at most $n-1$. If $R \le \varepsilon$ we can take $S=R$ so assume $R > \varepsilon$ from now on. For this purpose, we construct a simplicial map $$\phi: \Sd^N(\VR_R(X)^{(n)}) \to \VR_\varepsilon(X)$$ whose restriction to the vertices of $\VR_R(X)$ is the identity and which fits inside a square 
     \begin{center}
         \begin{tikzcd}
             \Sd^N(\VR_R(X)^{(n)}) \arrow[r, "\phi"] \arrow[d, "\gamma^N"] & \VR_\varepsilon(X) \arrow[d, "\subset"] \\
             \VR_R(X)^{(n)} \arrow[r, "\subset"] & \VR_S(X)
         \end{tikzcd},
     \end{center}
     that commutes up to homotopy after passing to geometric realizations. 

The construction of the map is inductive on skeleta. More precisely, we prove that, for each $i \in \{1,\ldots,n\}$, there exists $N_i \in \N$, and a simplicial map $\phi_i \colon \Sd^{N_i}(\VR_R(X)^{(i)}) \to \VR_{\varepsilon_i}(X)$ extending  $\phi_{i-1} \circ \gamma^{N_i - N_{i-1}}$ where  $$\phi_0: \VR_R(X)^{(0)} = X \overset{\mathrm{id}}{\longrightarrow} X = \VR_{\varepsilon_0}(X)^{(0)}$$ is the identity.

In order to define $\phi_1$ we consider the quotient map $\pi: G \to G/C = X$ and set $\bar 1_G = \pi(1_G)$. Note that $\pi$ is proper as $C$ is compact. Then, the preimage $V = \pi^{-1}(\bar{B_X}(\bar 1_G, \varepsilon_0))$ is a symmetric neighborhood of $1_G$. Because $G$ is connected, the neighborhood $V$ generates $G$.
    On the other hand, the preimage $K = \pi^{-1}(\bar{B_X}(\bar 1_G, R))$ is compact, and is therefore contained in $V^e$ for some $e \ge 1$.
    It follows that any two points $x$ and $y$ in $X$ with $d_X(x,y) \le R$ may be joined in $\VR_{\varepsilon_0}(X)$ by a path of length at most $e$. Letting $N_1$ be such that $2^{N_1} \ge e$, we can thus construct a simplicial map $\phi_1 \colon \Sd^{N_1}(\VR_R(X)^{(1)}) \to \VR_{\varepsilon_0}(X)$ extending the identity on vertices of $\VR_R(X)$, as the $N_1$-fold barycentric subdivision of an edge of $\VR_R(X)$ is merely a path graph of length $2^{N_1}$ with extremities in $\VR_R(X)$.

Let $i \in \{2,\ldots,n-1\}$ and assume that we have constructed a simplicial map $$\phi_{i-1}: \Sd^{N_{i-1}}(\VR_R(X)^{(i-1)}) \to \VR_{\varepsilon_{i-1}}(X).$$
    By Lemma \ref{lem:filling}, there exists $N_{i} \ge N_{i-1}$ such that, for every simplicial map $f \colon \Sd^{N_{i-1}}(\partial \Delta^{i}) \to \VR_{\varepsilon_{i-1}}(X)$, there is a simplicial map $\hat{f} \colon \Sd^{N_{i}}(\Delta^{i}) \to \VR_{\varepsilon_{i}}(X)$ with $\hat{f}|_{\Sd^{N_{i}}(\partial \Delta^{i})} = f \circ  \gamma^{N_{i} - N_{i-1}}_{\Sd^{N_{i-1}}(\partial \Delta^i)}$. Let $\sigma$ be $i$-simplex of $\VR_R(X)$. There exists a unique simplicial isomorphism $\theta \colon \Delta^i \to \langle \sigma \rangle$ that is order-preserving on vertices. Then $f_{\sigma} = \phi_{i-1} \circ \left(\Sd^{N_{i-1}}(\theta)|_{\Sd^{N_{i-1}}(\partial \Delta^i)}\right)$ is a simplicial map from  $\Sd^{N_{i-1}}(\partial \Delta^{i})$ to $VR_{\varepsilon_{i-1}}(X)$. Let $\hat f_{\sigma}\colon \Sd^{N_{i}}(\Delta^{i}) \to \VR_{\varepsilon_{i}}(X)$ be such that  $\hat{f_{\sigma}}|_{\Sd^{N_{i}}(\partial \Delta^{i})} = f_{\sigma} \circ   \gamma^{N_{i} - N_{i-1}}_{\Sd^{N_{i-1}}(\partial \Delta^i)}$. Finally, set $\phi_{i, \sigma} = \hat f_{\sigma} \circ \Sd^{N_i}(\theta^{-1})$. Then, the map $\phi_{i, \sigma}$ is simplicial from $\Sd^{N_i}(\langle \sigma \rangle)$ to $\VR_{\varepsilon_i}(X)$ and we have 
    $$\begin{array}{ccl}
        \phi_{i, \sigma}|_{\Sd^{N_i}(\langle \sigma \rangle \setminus \{\sigma\})} & = & \hat f_{\sigma}|_{\Sd^{N_i}(\partial \Delta^i)} \circ \Sd^{N_i}(\theta^{-1})|_{\Sd^{N_i}(\langle \sigma \rangle \setminus \{\sigma\})}\\
        & = & f_{\sigma} \circ \gamma^{N_i - N_{i-1}}_{\Sd^{N_{i-1}}(\partial \Delta^i)} \circ \Sd^{N_i}(\theta^{-1})|_{\Sd^{N_i}(\langle \sigma \rangle \setminus \{\sigma\})} \\
        & = & f_{\sigma} \circ \Sd^{N_{i-1}}(\theta^{-1})|_{\Sd^{N_{i-1}}(\langle \sigma \rangle \setminus \{\sigma\})} \circ \gamma^{N_i-N_{i-1}}_{\Sd^{N_{i-1}}(\langle \sigma \rangle \setminus \{\sigma\})} \\
        & = & \phi_{i-1} \circ \gamma^{N_i-N_{i-1}}_{\Sd^{N_{i-1}}(\langle \sigma \rangle \setminus \{\sigma\})}
    \end{array}.$$
    In particular, the maps $\phi_{i,\sigma}$ glue together to a well-defined map $$\phi_i: \Sd^{N_i}(\VR_R(X)^{(i)}) \to \VR_{\varepsilon_i}(X)$$ such that $\phi_i|_{\Sd^{N_i}(\VR_R(X)^{(i-1)})} = \phi_{i-1} \circ \gamma^{N_i - N_{i-1}}$. In particular, like $\phi_{i-1}$, the map $\phi_i$ induces the identity on vertices of $\VR_R(X)$, which completes the induction.

We have thus constructed a simplicial map $$\phi: \Sd^N(\VR_R(X)^{(n)}) \to \VR_\varepsilon(X)$$ extending the identity map on vertices of $\VR_R(X)$.

Now we put $S = R+d\varepsilon$ where $d$ is the diameter of $\Sd^N(\Delta^{n})$ and claim that the following diagram induces a homotopy commutative square when passing to geometric realizations
     \begin{center}
         \begin{tikzcd}
             \Sd^N(\VR_R(X)^{(n)}) \arrow[r, "\phi"] \arrow[d, "\gamma^N"] & \VR_\varepsilon(X) \arrow[d, "\subset"] \\
             \VR_R(X)^{(n)} \arrow[r, "\subset"] & \VR_S(X)
         \end{tikzcd}.
     \end{center}
     It suffices to check that $\phi$ and $\gamma^N$ are contiguous as maps to $\VR_S(X)$. Note that, if $\sigma$ is a simplex of $\Sd^N(\VR_R(X)^{(n)})$, then it lies in $\Sd^N(\langle \tau\rangle)$ for a simplex $\tau$ of $\VR_R(X)^{(n)}$. By definition of $\gamma^N$, we must have $\gamma^N(\sigma) \subseteq \tau$. 
     On the other hand, note that any two vertices of $\Sd^N(\tau)$ are connected by a path of length at most $d$. In particular, any vertex of $\tau$ is connected to any vertex of $\sigma$ by a path of length at most $d$ in $\Sd^N(\VR_R(X)^{(n)})$. As $\phi$ is simplicial and fixes the vertices of $\VR_R(X)^{(n)}$ (and in particular, the vertices of $\tau$), it follows that any vertex of $\phi(\sigma)$ is connected to any vertex of $\tau$ by a path of length at most $d$ in $\VR_\varepsilon(X)$. In particular, for any $x \in \gamma^N(\sigma)$ and $y \in \phi(\sigma)$, we have $d_X(x,y) \le d\varepsilon$. It follows that $\phi(\sigma) \cup \gamma^N(\sigma)$ is a simplex in $\VR_S(X)$, which completes the proof. 

     As $|\gamma^N|$ is a homotopy equivalence, it follows that the inclusion $$|\VR_R(X)^{(n)}| \to |\VR_S(X)|$$ factors through the contractible space $|\VR_\varepsilon(X)|$ up to homotopy, and is thus homotopically trivial. Consequently, $\pi_i\big(|\VR_R(X)| \to |\VR_S(X)|\big)$ is trivial for $0 \le i \le n-1$. 
\end{proof}

\section{Proof}

To prove Theorem~\ref{thm:main} it remains to treat the tdlc case and deal with extensions.

\begin{proposition}\label{prop:tdlc}
  For a locally compact $G$ the following conditions are equivalent:
  \begin{enumerate}
    \item $G$ is of type $C_n$ in the sense of Abels--Tiemeyer.\label{item:tdlc_cn}
    \item $G$ is coarsely $(n-1)$-connected for $\calK$.\label{item:tdlc_cc}
  \end{enumerate}
  If $G$ is totally disconnected they are also equivalent to:
  \begin{enumerate}
    \setcounter{enumi}{2}
    \item $G$ is of type $F_n$ in the sense of Castellano--Corob-Cook.\label{item:tdlc_fn}
  \end{enumerate}
\end{proposition}

\begin{proof}
  The equivalence between \eqref{item:tdlc_cn} and \eqref{item:tdlc_cc} is Corollary~\ref{cor:cn_coarse_connected}

   Now assume that $G$ is tdlc and let $C < G$ be a compact open subgroup (which exists by van Dantzig's theorem). Then $T = G/C$ is discrete and $G$ acts on $T$ properly and cocompactly. Let $X = |ET|$ be the realization of the free simplicial set and, for $F \subseteq T$ finite, let $X_F = |G \cdot EF|$. Then $(X_F)_F$ is a filtration of $X$ by $G$-cocompact subcomplexes. By Proposition~\ref{prop:at_space} $G$ is of type $C_n$ if and only if $(\pi_i(X_F))_F$ is essentially trivial for $i < n$. By Theorem~\ref{thm:cc_brown} $G$ is of type $F_n$ if and only if the same is true. This shows \eqref{item:tdlc_cn} $\Leftrightarrow$ \eqref{item:tdlc_fn}.
\end{proof}

The following result due to Bux, Hartmann, and Quintanilha shows that compactness properties $C_n$ behave under extensions like finiteness properties of discrete groups. It is the special case $\chi = 0$ of \cite[Theorem~H]{BuxHartmannQuintanilha}:

\begin{theorem}\label{thm:extension}
  Let $1 \to N \to G \to Q \to 1$ be an extension of locally compact groups and let $N$ be of type $C_n$.
  \begin{enumerate}
    \item If $G$ is of type $C_{n+1}$ then $Q$ is of type $C_{n+1}$.
    \item If $Q$ is of type $C_n$ then $G$ is of type $C_n$.
  \end{enumerate}
\end{theorem}

\begin{proof}[Proof of Theorem~\ref{thm:main}]
  The equivalence of \eqref{item:cn} and \eqref{item:cc} was shown in Proposition~\ref{prop:tdlc}. By Theorem~\ref{thm:connected} $G^0$ is of type $C_\infty$ so by Theorem~\ref{thm:extension} $G$ is of type $C_n$ if and only if $G/G^0$ is (proving Corollary~\ref{cor:main} for property $C_n$). Now Proposition~\ref{prop:tdlc} applied to $G/G^0$ completes the proof.
\end{proof}

\begin{remark}
  If $R$ is a ring, the same reasoning (using the appropriate results in \cite{CastellanoCorobCook20} and \cite{AbelsTiemeyer97}) shows that for tdlc groups $\mathit{CP}_n(R)$ (which Abels--Tiemeyer only define for $R = \Z$ but generalizes in an obvious way) is equivalent to $\mathit{FP}_n(R)$ and to being coarsely $(n-1)$-acyclic for $\calK$. Since $\mathit{C}_\infty$ implies  $\mathit{CP}_\infty(R)$, connected groups enjoy that property as well. What is not (yet) available is a cohomological version of Theorem~\ref{thm:extension}.
\end{remark}

\bibliographystyle{amsalpha}
\bibliography{lc_top_fin}

\end{document}